\documentclass{article}

\usepackage{amssymb}
\usepackage{amsthm}
\usepackage[curve,matrix,arrow]{xy}
\theoremstyle{plain}\newtheorem{Theorem}{Theorem}
\theoremstyle{plain}
\theoremstyle{plain}
\theoremstyle{plain}\newtheorem{Lemma}[Theorem]{Lemma}
\theoremstyle{plain}\newtheorem{Proposition}[Theorem]{Proposition}
\theoremstyle{definition}
\theoremstyle{definition}
\theoremstyle{definition}
\theoremstyle{definition}\newtheorem{Remark}[Theorem]{Remark}

\def\N{{\mathbb N}}

           \def\tenk{\otimes_k}     
\def\Br{\mathrm{Br}}             

\def\dim{\mathrm{dim}}           \def\tenkP{\otimes_{kP}}

\def\Id{\mathrm{Id}}             
\def\Im{\mathrm{Im}}             
\def\Ind{\mathrm{Ind}}

\def\Res{\mathrm{Res}}           
           
\def\Tr{\mathrm{Tr}}             
\def\tr{\mathrm{tr}}

\title{Bounds for Hochschild cohomology of block algebras} 
\author{Radha Kessar and Markus Linckelmann} 
\date{\today}

\begin{document}

\maketitle

\begin{abstract} 
We show that for any block algebra $B$ of a finite group over 
an algebraically closed field of prime characteristic $p$ the
dimension of $HH^n(B)$ is bounded by a function depending only on
the nonnegative integer $n$ and the defect of $B$. The proof uses
in particular a theorem of Brauer and Feit which implies the result
for $n=0$.
\end{abstract}

Let $p$ be a prime and $k$ an algebraically closed field of 
characteristic $p$. Let $G$ be a finite group and $B$ a block 
algebra of $kG$; that is, $B$ is an indecomposable direct factor 
of $kG$ as a $k$-algebra. A {\it defect group of $B$} is a minimal 
subgroup $P$ of $G$ such that $B$ is isomorphic to a direct 
summand of $B\tenkP B$ as a $B$-$B$-bimodule. The defect groups of 
$B$ form a $G$-conjugacy class of $p$-subgroups of $G$, and the 
{\it defect of} $B$ is the integer $d(B)$ such that $p^{d(B)}$ is
the order of the defect groups of $B$. The {\it weak Donovan 
conjecture} states that the Cartan invariants of $B$ are bounded
by a function depending only on the defect $d(B)$ of $B$. As a 
consequence of a theorem of Brauer and Feit \cite{BrFe59}, the 
number of isomorphism classes of simple $B$-modules is bounded 
by a function depending only on $d(B)$. 
Thus the weak Donovan conjecture would imply that the dimension of 
a basic algebra of $B$ is bounded by a function depending on $d(B)$. 
This in turn would imply that the dimension of the term in any fixed 
degree $n$  of the Hochschild complex of a basic algebra of $B$ is 
bounded by a function depending on $n$ and $d(B)$; since Hochschild 
cohomology is invariant under Morita equivalences, we would thus get 
that the dimension of $HH^n(B)$ is bounded by a function depending 
on $n$ and $d(B)$. The purpose of this note is to show that this 
consequence of the weak Donovan conjecture does indeed hold. 

\begin{Theorem} \label{HHbound}
There is a function $f : \N_0\times \N_0\to \N_0$
such that for any integer $n\geq 0$, any finite group $G$
and any block algebra $B$ of $kG$ with defect $d$  we have
$$\dim_k(HH^n(B)) \leq f(n,d)$$
\end{Theorem}

For $n=0$ this follows from the aforementioned theorem of Brauer 
and Feit \cite{BrFe59}, since $HH^0(B)\cong$ $Z(B)$. 
Using Tate duality, the theorem above extends to Tate cohomology for
negative $n$. 
A result of K\"ulshammer and Robinson \cite[Theorem 1]{KuRo}
implies that it suffices to show theorem \ref{HHbound} for
finite groups with a non-trivial normal $p$-subgroup. We
follow a slightly different strategy in the proof below,
reducing the problem directly to finite groups with a non-trivial 
central $p$-subgroup.

\begin{Remark} \label{fdef}
We make no effort to construct a best possible bound; 
we define the function $f$ in theorem \ref{HHbound} inductively 
as follows: we set $f(0,0)=1$, $f(n,0)=0$ for $n>0$;
for all $d > 0$, $f(0,d)$ is the largest integer less or equal to
the bound $\frac{1}{4}p^{2d}+1$ given in the Brauer-Feit theorem,
and for $n>0$, $d>0$ we set
$$f(n,d) = p\cdot c(d)\cdot \sum_{i=0}^n\ f(i,d-1)$$
where $c(d)$ is the maximum of the numbers of subgroups in any
finite group of order $p^d$.
\end{Remark}

Let $G$ be a finite group and $U$ a $kG$-module. We denote
as usual by $U^G$ the subspace of $G$-fixed points in $U$.
If $H$ is a subgroup of $G$ then $U^G\subseteq$ $U^H$,
and there is a {\it trace map} $\tr^G_H : U^H\to$ $U^G$
sending $u\in$ $U^H$ to $\sum_{x\in[G/H]}xu$, where $[G/H]$
is a set of representatives of the $H$-cosets in $G$; one checks
that this map is independent of the choice of $[G/H]$ and that
its image, denoted $U^G_H$, is contained in $U^G$. For
$Q$ a $p$-subgroup of $G$, we denote the Brauer construction
of $U$ with respect to $Q$ by $U(Q)=$ $U^Q/\sum_{R; R<Q}\ U^Q_R$
and by $\Br_Q^U:U^Q\to$ $U(Q)$ the canonical surjection, called
{\it Brauer homomorphism}. A block algebra $B$ of $kG$ can be
viewed as an indecomposable $k(G\times G)$-module, with
$(x,y)\in$ $G\times G$ acting by left multiplication with $x$
and right multiplication with $y^{-1}$. For $H$ a subgroup of $G$, we
denote by $\Delta H$ the `diagonal' subgroup 
$\Delta H=$ $\{(h,h)\ |\ h\in H\}$ in $G\times G$. 
In particular, the action of $\Delta G$ on $B$ can be
identified with the conjugation action of $G$ on $B$.
The Brauer construction applied to $B$ with respect
to $\Delta Q$ is canonically isomorphic to $kC_G(Q)c$, where
$Q$ is a $p$-subgroup of $G$ and $c=$ $\Br_{\Delta Q}(1_B)$. 
A $B$-{\it Brauer pair} is a pair $(Q,e)$ consisting of a 
$p$-subgroup $Q$ of $G$ and of a block idempotent $e$ of
$kC_G(Q)$ satisfying $e\Br_Q(1_B)=$ $e$. The set of $B$-Brauer
pairs is a $G$-poset in which the maximal pairs are all conjugate.
The maximal $B$-Brauer pairs are exactly the $B$-Brauer pairs
$(Q,e)$ for which $Q$ is a defect group of $B$.
See \cite{AlBr} and \cite[\S 11, \S 40]{Thev} for details.
In what follows we use without further comment the canonical
graded isomorphism $HH^*(B)\cong$ $H^*(\Delta G;B)$; see
\cite[(3.2)]{SiWi}.
The following result is certainly well-known but not
always stated in exactly the form we need it; we therefore
give a proof for the convenience of the reader.

\begin{Proposition} \label{blocksummands}
Let $G$ be a finite group, $B$ be a block algebra of $kG$
and $Q$ a $p$-subgroup of $G$. Set $b=$ $1_B$ and $c=$
$\Br_Q(b)$. Suppose that $c\neq 0$ and set $B_Q=$
$kC_G(Q)cb$. Then we have a direct sum decomposition of 
$kN_{G\times G}(\Delta Q)$-modules
$$\Res^{G\times G}_{N_{G\times G}(\Delta Q)}(B) = B_Q\oplus C_Q$$
such that multiplication by $b$ is an isomorphism of
$kN_{G\times G}(\Delta Q)$-modules $kC_G(Q)c\cong$ $B_Q$
and such that $C_Q(\Delta Q)=$ $\{0\}$.
\end{Proposition}

The proof we present here uses the following well-known
lemma, which is a special case of expressing relative projectivity 
in terms of the splitting of adjunction maps (the general theme 
behind this is developed in \cite{Broue09}, \cite{Chou}, for 
instance).

\begin{Lemma} \label{adjunitsplit}
Let $\alpha : B\to$ $A$ be a homomorphism of $k$-algebras.
Suppose that $B$ is isomorphic to a direct summand of $A$ as a 
$B$-$B$-bimodule. Then $\alpha$ is injective and $\Im(\alpha)$ 
is a direct summand of $A$ as a $B$-$B$-bimodule.
\end{Lemma}

\begin{proof}
The left or right action of an element $b\in$ $B$ on $A$
is given by left or right multiplication with $\alpha(b)$.
Let $\iota : B\to$ $A$ and $\pi : A\to$ $B$ be $B$-$B$-bimodule
homomorphisms satisfying $\pi\circ\iota=$ $\Id_B$. Then
$\iota(1_B)$ commutes with $\Im(\alpha)$, the map
$\beta$ sending $a\in$ $A$ to $a\iota(1_B)$ is an
$A$-$B$-bimodule endomorphism of $A$, and we have
$\beta(\alpha(b))=$ $\alpha(b)\iota(1_A)=$ $\iota(b)$, hence
$\beta\circ\alpha=$ $\iota$.   
Thus $\pi\circ\beta\circ\alpha=$ $\Id_B$, which shows that
as a $B$-$B$-bimodule homomorphism, $\alpha$ is split injective with 
$\pi\circ\beta$ as a retraction.
\end{proof}

\begin{proof}[Proof of Proposition \ref{blocksummands}]
For any block of $kN_G(Q)$ which appears in $kN_G(Q)c$, the
block $B$ of $kG$ is the corresponding `induced' block. 
By \cite[\S 14, Lemma 1]{Alpbook}, 
$kN_G(Q)c$ is isomorphic to a direct
summand of $B$ as a $kN_G(Q)$-$kN_G(Q)$-bimodule, and thus of $cBc$, 
as a $kN_G(Q)c$-$kN_G(Q)c$-bimodule. 
By lemma \ref{adjunitsplit}, multiplication by $b$ induces an
algebra homomorphism  $kN_G(Q)c\to$ $cBc$ which is split injective 
as a homomorphism of $kN_G(Q)c$-$kN_G(Q)c$-bimodules.
Since $kC_G(Q)c$ is a direct summand of
$kN_G(Q)c$ as an $N_{G\times G}(\Delta Q)$-module we get that
$kC_G(Q)c\cong$ $B_Q$ and that $B_Q$ is a direct summand
of $B$ as an $N_{G\times G}(\Delta Q)$-module.
Moreover, $B(\Delta Q)\cong$ $B_Q$, 
and hence any complement $C_Q$ of $B_Q$ in $B$, as an 
$N_{G\times G}(\Delta Q)$-module, satisfies $C_Q(\Delta Q)=$ $\{0\}$.
\end{proof}

We will make use of the following well-known fact on transfer in
cohomology (we include a short proof for the convenience of the 
reader).

\begin{Lemma} \label{transferinduced}
Let $G$ be a finite group, $H$ a subgroup of $G$ and $V$ a
$kH$-module. Let $U$ be a direct summand of $\Ind^G_H(V)$.
Then $H^*(G; U) =$  $\tr^G_H(H^*(H; \Res^G_H(U)))$.
\end{Lemma}

\begin{proof} By Higman's criterion there is a $kH$-endomorphism
$\varphi$ of $U$ such that $\Id_U=$ $\tr^G_H(\varphi)$. Let
$n\geq 0$ and let $\zeta : \Omega^n(k)\to$ $U$ be a $kG$-homomorphism,
representing an element in $H^n(G;U)$. Then $\zeta=$ 
$\Id_U\circ\zeta=$ $\tr^G_H(\varphi\circ\zeta)$, whence the result.
\end{proof}

This is applied in the following situation:

\begin{Lemma} \label{blocktransferinduced}
Let $G$ be a finite group, $B$ a block algebra of $kG$ and
$P$ a defect group of $B$. We have $H^*(\Delta G;B)=$
$\tr^{\Delta G}_{\Delta P}(H^*(\Delta P; B))$.
\end{Lemma}

\begin{proof} 
As a $k(G\times G)$-module, $B$ has vertex 
$\Delta P$ and trivial source, thus is isomorphic to a direct
summand of $\Ind^{G\times G}_{\Delta P}(k)$. Mackey's formula
shows that $\Res^{G\times G}_{\Delta G}(B)$ is still relatively
$\Delta P$-projective, hence lemma \ref{transferinduced} implies
the result. Alternatively, this follows from the fact that
$b=1_B$ can be written as a relative trace of the form
$b=$  $\Tr^{\Delta G}_{\Delta P}(y)$ for some $y\in$ $B^{\Delta P}$.
\end{proof}

\begin{Proposition} \label{HHsumtransfer}
Let $G$ be a finite group and $B$ be a block algebra of $kG$.
Set $b=$ $1_B$ and for every $B$-Brauer pair $(Q,e)$ set $B_{(Q,e)}=$
$kC_G(Q)eb$. Then $B_{(Q,e)}$ is a direct summand of $B$ as
a $k(C_G(Q)\times C_G(Q))\Delta Q$-module, isomorphic to
$kC_G(Q)e$. In particular, $H^*(\Delta Q;B_{(Q,e)})$ is a direct
summand, as a graded vector space, of $H^*(\Delta Q;B)$, and we have
$$H^*(\Delta G; B)= \sum_{(Q,e)}\ \ 
\tr^{\Delta G}_{\Delta Q}(H^*(\Delta Q; B_{(Q,e)}))$$
where in the sum $(Q,e)$ runs over a set of representatives of the
$G$-conjugacy classes of $B$-Brauer pairs.
\end{Proposition}

\begin{proof}
The proof adapts techniques that have been used in the proof 
of a result of Watanabe \cite[Lemma 1]{Wat}.
Clearly $H^*(\Delta G;B)$ contains the right side in the displayed
equation. We need to show that $H^*(\Delta G;B)$ is contained
in the right side. Since any summand of the right side of the
form $\tr^{\Delta G}_{\Delta Q}(H^*(\Delta Q; B_{(Q,e)})$ 
depends only on the $G$-conjugacy class of $(Q,e)$ it
suffices to prove the inclusion
$$H^*(\Delta G;B) \subseteq \sum_Q
\tr^{\Delta G}_{\Delta Q}(H^*(\Delta Q; B_Q))$$
where $Q$ runs over the $p$-subgroups of $G$ for which
$\Br_Q(b)\neq 0$. Note that this makes sense since $B_Q$ is
a direct summand of $B$ as a $k\Delta Q$-module, hence
$H^*(\Delta Q; B_Q)$ is a subspace of $H^*(\Delta Q; B)$, to which
we then apply the transfer map $\tr^{\Delta G}_{\Delta Q}$.
Since $H^*(\Delta G;B)=$ $\tr^{\Delta G}_{\Delta P}(H^*(\Delta P;B))$
by lemma \ref{blocktransferinduced} it suffices to show that
the right side contains $\tr^{\Delta G}_{\Delta R}(H^*(\Delta R;B))$
for any $p$-subgroup $R$ of $G$. This will be shown by induction.
For $R=$ $\{1\}$ this holds trivially because $B_{\{1\}}=$ $B$
and $C_{\{1\}}=$ $\{0\}$. For $R\neq\{1\}$ we have a direct sum
decomposition $B=$ $B_R\oplus C_R$ of $kN_{G\times G}(\Delta R)$-modules
as in proposition \ref{blocksummands}, and hence
$$H^*(\Delta R;B) = H^*(\Delta R;B_R) + H^*(\Delta R;C_R)$$
Since $C_R(\Delta R)=$ $\{0\}$ we have
$$H^*(\Delta R;C_R)\subseteq \sum_{S; S<R}\ \ 
\tr^{\Delta R}_{\Delta S}(H^*(\Delta S;B))$$
by lemma \ref{transferinduced}.
Applying the transfer map $\tr^{\Delta G}_{\Delta R}$ yields
$$\tr^{\Delta G}_{\Delta R}(H^*(\Delta R;C_R))\subseteq
\sum_{S; S<R}\ \ \tr^{\Delta G}_{\Delta S}(H^*(\Delta S;B))$$
hence
$$\tr^{\Delta G}_{\Delta R}(H^*(\Delta R; B)) \subseteq
\tr^{\Delta G}_{\Delta R}(H^*(\Delta R; B_R)) +
\sum_{S; S<R}\ \ \tr^{\Delta G}_{\Delta S}(H^*(\Delta S;B))$$
The result follows by induction.
\end{proof}

\begin{Lemma} \label{HHQe}
Let $G$ be a finite group and $B$ be a block algebra of $kG$.
Set $b=$ $1_B$ and for every $B$-Brauer pair $(Q,e)$ set $B_{(Q,e)}=$
$kC_G(Q)eb$. For any integer $n\geq 0$ we have
$$\dim_k(\tr_{\Delta Q}^{\Delta G}(H^n(\Delta Q; B_{(Q,e)})))\leq 
\dim_k(H^n(\Delta QC_G(Q); kQC_G(Q)e))$$
\end{Lemma}

\begin{proof} Clearly 
$\dim_k(\tr_{\Delta Q}^{\Delta G}(H^n(\Delta Q; B_{(Q,e)})))\leq$ 
$\dim_k(\tr_{\Delta Q}^{\Delta QC_G(Q)}(H^n(\Delta QC_G(Q); B_{(Q,e)})))$.
Moreover, since $B_{(Q,e)}\cong$ $kC_Q(Q)e$ is isomorphic to a 
direct summand of $kQC_G(Q)e$, the lemma follows.
\end{proof}

\begin{Lemma} \label{HHZ}
Let $G$ be a finite group, $B$ a block of $kG$ and $Z$ a
subgroup of order $p$ of $Z(G)$. Set $\bar G=$ $G/Z$ and
denote by $\bar B$ the image of $B$ in $k\bar G$ under the
canonical algebra homomorphism $kG\to$ $k\bar G$. For
any integer $n\geq 0$ we have
$$\dim_k(H^n(\Delta G; B))\leq
p\cdot\sum_{i=0}^{n}\ \dim_k(H^i(\Delta \bar G; \bar B))$$
\end{Lemma}

\begin{proof}
The Lyndon-Hochschild-Serre spectral sequence associated with
$G$, $Z$, $\bar G$ and $B$ endowed with the conjugation action
of $G$ reads
$$H^i(\Delta\bar G; H^j(\Delta Z; B)) \Rightarrow H^{i+j}(\Delta G; B)$$
Since $\Delta Z$ acts trivially on $kG$, hence on $B$, 
we have $H^j(\Delta Z;B)\cong$ $H^j(\Delta Z;k)\tenk B\cong$ $B$,
where the last isomorphism uses that we have
$H^j(\Delta Z;k)\cong$ $k$ because $Z$ is cyclic. 
Thus $H^n(\Delta G;B)$ is filtered by
subquotients of $H^i(\Delta\bar G; B))$, with $0\leq i\leq n$; in
particular, 
$\dim_k(H^n(\Delta G; B))\leq$
$\sum_{i=0}^{n}\ \dim_k(H^i(\Delta \bar G; B))$.
Let $z$ be a generator of $Z$. As a $k\Delta \bar G$-module,
$B$ has a filtration of the form
$$B\supseteq B(1-z)\supseteq B(1-z)^2\supseteq
\cdots\supseteq B(1-z)^{p-1}\supseteq \{0\}$$
and since $B$ is projective as a right $kZ$-module, the
quotient of any two consecutive terms in this filtration
is isomorphic to $\bar B$. Thus the appropriate long exact
sequences in cohomology imply that
$\dim_k(H^i(\Delta \bar G; B))\leq$
$p\cdot\dim_k(H^i(\Delta \bar G; \bar B))$, whence the result.
\end{proof}

\begin{proof}[Proof of Theorem \ref{HHbound}]
Let $f$ be the function defined in remark \ref{fdef}.
Note that $f(n,d)\geq$ $f(n,d-1)$ for all $n\geq 0$
and all $d>0$. Denote by $c(d)$ the maximum 
of the numbers of subgroups in finite groups of order $p^d$.
As mentioned before, theorem \ref{HHbound} holds for $n=0$.
Clearly theorem \ref{HHbound} holds  
for $d=0$ because a defect zero block is a matrix algebra.
Let $n$ and $d$ be a positive integers. 
Then $\tr_{\Delta 1}^{\Delta G}(H^n(1; B))=$ $\{0\}$. Thus, by
proposition \ref{HHsumtransfer} and lemma \ref{HHQe}
we have $\dim_k(HH^n(B))\leq$ $\sum_{(Q,e)}\ \dim_k(HH^n(QC_G(Q)e))$
where in the sum $(Q,e)$ runs over a set of representatives
of the $G$-conjugacy classes of non-trivial $B$-Brauer pairs.
Any such pair $(Q,e)$ has a conjugate with $Q$ contained in a fixed
defect group $P$, and hence the number of summands in this
sum is at most $c(d)$. Moreover, $Z(QC_G(Q))$ contains $Z(Q)$, and
hence $QC_G(Q)$ has a non-trivial central subgroup $Z_Q$
of order $p$. After replacing $(Q,e)$ by a suitable $G$-conjugate,
we may assume that $QC_P(Q)$ is a defect group of $e$ viewed as
a block of $kQC_G(Q)$; in particular the defect groups of $e$
have order at most $|P|=$ $p^d$. Thus the defect groups of the image
$\bar e$ of $e$ in $kQC_G(Q)/Z_Q$ have order at most $|P|/p=$ $p^{d-1}$,
hence $\dim_k(HH^n(kQC_G(Q)/Z_Q\bar e))\leq$ $f(n,d-1)$.
It follows from lemma \ref{HHZ} that $\dim_k(HH^n(kQC_G(Q)e))\leq$
$p\cdot\sum_{i=0}^n\ f(i,d-1)$. Together with the above remarks
we get the inequality $\dim_k(HH^n(B)\leq$ 
$p\cdot c(d)\cdot\sum_{i=0}^n\ f(i,d-1)=$ $f(n,d)$, as required. 
\end{proof}

\begin{Remark}
The strong version of Donovan's conjecture states that 
for a fixed integer $d\geq 0$ there  should be only finitely 
many Morita equivalence classes of blocks with defect at most 
$d$. If true, this would imply that there are only finitely 
many isomorphism classes of Hochschild cohomology algebras of 
blocks with defect at most $d$; this remains an open problem.
\end{Remark}


\bigskip

\begin{thebibliography}{WWW}

\bibitem{Alpbook} J. L. Alperin, {\em Local representation theory},
Cambridge studies in advanced mathematics {\bf 11}, Cambridge
University Press (1986).

\bibitem{AlBr} J. L. Alperin and M. Brou\'e, Local methods
in block theory. \emph{Ann. of Math.} \textbf{ 110} (1979)
143--157.




\bibitem{BrFe59} R. Brauer and W. Feit, {\em On the number of
irreducible characters of finite groups in a given block},
Proc. Nat. Acad. Sci. U.S.A. {\bf 45} (1959), 361--365.


\bibitem{Broue09} M. Brou\'e, {\em Higman's criterion revisited},
Michigan Math. J. {\bf 58} (2009), 125--179.

\bibitem{Chou} L. G. Chouinard, {\em Transfer maps}, Comm. Alg.
 {\bf 8} (1980), 1519--1537.



\bibitem{KuRo} B. K\"ulshammer and G. R. Robinson, {\em An
alternating sum for Hochschild cohomology of a block},
J. Algebra {\bf 249} (2002), 220--225.



\bibitem{SiWi} S. F. Siegel, S. Witherspoon, {\em The Hochschild
cohomology ring of a group algebra}, Proc. London Math. Soc.
{\bf 79} (1999), 131--157.


\bibitem{Thev} J. Th\'evenaz, {\em $G$-Algebras and Modular Representation
Theory}, Oxford Science Publications, Clarendon Press, Oxford (1995).

\bibitem{Wat} A. Watanabe, {\em Note on a $p$-block of a finite group
with abelian defect group}, Osaka J. Math. {\bf 26} (1989), 829--836.

\end{thebibliography}
\end{document}